\documentclass[12pt]{amsart}
\usepackage{amssymb,amsmath,amscd}
\usepackage{graphicx}
\usepackage{mathrsfs}
\usepackage{pgf,tikz}
\newtheorem{thm}{\bf Theorem}[section]
\newtheorem{prop}[thm]{\sc Proposition}

\newtheorem{cor}[thm]{\sc Corollary}
\theoremstyle{definition}\newtheorem{exa}[thm]{\sc Example}
\theoremstyle{definition}
\theoremstyle{definition}\newtheorem{rem}[thm]{\sc Remark}
\theoremstyle{definition}
\theoremstyle{definition}
\theoremstyle{definition}

\numberwithin{equation}{section}

\DeclareMathOperator{\R}{\mathbb R}

\DeclareMathOperator{\C}{\mathbb C}

\begin{document}
\title[sample]{The Classification of Rotationally symmetric hypersurfaces in the Heisenberg groups $H_{n}$}
\author{Hung-Lin Chiu, Sin-Hua Lai and Hsiao-Fan Liu}
\address{Department of Mathematics, National Tsing Hua University, Hsinchu, Taiwan 300, R.O.C.; and National Center for Theoretical Sciences, Taipei, Taiwan}
\email{hlchiu@math.nthu.edu.tw}
\address{Fundamental Education Center, National Chin-Yi University of
Technology,Taichung 41170, Taiwan, R.O.C.}
\email{shlai@ncut.edu.tw}
\address{Department of Applied Mathematics, National Chung Hsing University, Taichung,Taiwan, R.O.C.}
\email{hfliu@nchu.edu.tw}


\begin{abstract}
 In this paper, we show the fundamental theorems for rotationally symmetric hypersurfaces, and thus, together with the earlier results in \cite{CCHY1} and \cite{CCHY2}, provide a complete classification of umbilic hypersurfaces in the Heisenberg groups $H_{n}$. In addition, we  give a complete description of generating curves for rotationally symmetric hypersurfaces with constant $p$-mean curvature $H=c$ (including $H=0$) in the Heisenberg group $H_{n}$.
We also establish the validity of Alexandrov's theorem for rotationally symmetric hypersurfaces in $H_n$.   
\end{abstract}

\maketitle
\tableofcontents
\section{Introduction}
Pseudohermitian geometry (\cite{DT,J,W,Y}) came from several complex variables. It is a branch of Cartan geometry (\cite{PT,S}). The corresponding Klein geometry is the Heisenberg group $H_{n}$ (\cite{CCHY1,CCHY2,CFH,CL,CHL}). That is, any pseudohermitian manifold with vanishing torsion and curvature is locally part of the Heisenberg group. In this sense, the Heisenberg group in Pseudohermitian geometry plays the same role as the Euclidean space plays in Riemannian geometry. 

The classical topic to study in the geometry of submanifolds is the umbilic hypersurfaces of Euclidean space. In the Heisenberg group, there is also a similar umbilicity concept, which was first introduced in \cite{CCHY1} by J.-H. Cheng, H.-L. Chiu, J.-F. Hwang, and P. Yang,  to study Alexandrov's theorem, which states that, in the Heisenberg group $H_{n}$, the only connected, closed, $C^2$ constant $p$-mean curvature hypersurfaces are the Pansu spheres. For the definition of umbilicity and related concepts, we refer the readers to \cite{CCHY1} and \cite{CCHY2}, or Section \ref{umhys}. There are two kinds of eigenvalues, $k$ and $l$, of the shape operator. The eigenvalue $l$ is just the $p$-curvature of the characteristic curves. For the Pansu sphere, we have that $l=2k$. 

Let $PSH(n)$ be the symmetric group of the Heisenberg rigid motions on $H_{n}$ with CR-dimension $n$, and $\Sigma$ a hypersurface in $H_{n}$. If $\Sigma$ is invariant under the action of the subgroup $U(n)\subset PSH(n)$, then we say that the surface $\Sigma$ is rotationally symmetric. If the function $\alpha$ of $\Sigma$ vanishes, i.e., $\alpha=0$ at each point of the rotationally symmetric hypersurface $\Sigma$, then $\Sigma$ is a cylinder. On the other hand, on the region where $\alpha\neq 0$, it can be defined by the graph of a function depending only on the radius $r$. For $n\geq 2$, in Proposition 3.1 in \cite{CCHY1}, we showed that each rotationally symmetric hypersurface is an umbilic hypersurface. Actually we have\\

{\bf Proposition} ({\bf Proposition 3.1}. in \cite{CCHY1}) If $\Sigma$ is rotationally symmetric, then it is umbilic with $\alpha^{2}+k^{2}=\frac{1}{x^2}>0$. If, in addition, it is closed
and satisfies the condition $l=2k$, then $\Sigma$ must be the Pansu sphere $S_{\lambda}$ with $\lambda=k$.\\

Conversely, we have\\

{\bf Theorem} ({\bf Theorem B}  in \cite{CCHY2})  Let $\Sigma$ be an immersed, connected, umbilic hypersurface in $H_{n}, n\geq 2$. 
Let $S_{\Sigma}$ denote the set of all singular points in $\Sigma$. Then \\
(a) either $\alpha^{2}+k^{2}\equiv 0$ on $\Sigma$ or $\alpha^{2}+k^{2}>0$ at all points in $\Sigma\backslash S_{\Sigma}$.\\
(b) Suppose $\alpha^{2}+k^{2}\equiv 0$ on $\Sigma$. Then $\Sigma$ is congruent with part of a hypersurface $C^{n-1}\times\gamma\times\R$, where $\gamma$ is a curve in the Euclidean plane $C$ with signed curvature $l(=H)$.\\
(c) Suppose $\alpha^{2}+k^{2}>0$ at all points in $\Sigma\backslash S_{\Sigma}$. Then $\Sigma$ is congruent with part of a rotationally invariant (symmetric) hypersurface. Moreover, the radius of leaves in the associated foliation is $\frac{1}{\sqrt{\alpha^{2}+k^{2}}}$.\\

Interestingly, there is another geometric meaning for the eigenvalue $k$. The meaning is also true for $n=1$. Actually, the value $k$ can be realized as the derivative of a kind of angle function. This angle function makes the generating curve a horizontal generating curve. For the details of this geometric meaning of $k$, we refer the reader to Proposition \ref{gemeok} in Section \ref{sorss}. Let $\Sigma$ be a rotationally symmetric hypersurface in $H_{n}$ with generating curve 
\[(0,\cdots,0,x(s),0,\cdots,0,0,t(s)) \] on the $x_{n}t$-plane, then using the horizontal arc-length $\tilde{s}$ (for the definition, see \cite{CFH,CHL}), we have (see \eqref{goinf3})
\begin{equation*}
\begin{split}
x(\tilde{s})&=\int \alpha x(\tilde{s})d\tilde{s}\ \ \left(\Leftrightarrow x(\tilde{s})=e^{\int\alpha d\tilde{s}}\right)\\
t(\tilde{s})&=-\int kx^2(\tilde{s})d\tilde{s},
\end{split}
\end{equation*}
with \[x(\tilde{s})^{2}=\frac{1}{\alpha^2+k^{2}}(\tilde{s}),\] where $\tilde{s}$ is the horizontal arc-length of the corresponding horizontal generating curve.
Conversely, for rotationally symmetric hypersurfaces, we have the following fundamental theorem, which shows that the two functions $\alpha$ and $k$ constitute a complete set of invariants. This is the first main theorem in this paper.\\

{\bf Theorem I.}
Given two functions $\bar\alpha(s),\bar{k}(s)$ which satisfy the integrability condition \eqref{nomal}.
Then there exists a rotationally symmetric hypersurface in $H_{n}$ such that $\alpha(s)=\bar\alpha(s)$ and $k(s)=\bar k(s)$ for all $s$, and $s$ is the horizontal arc-length. In addition, such a hypersurface is unique, up to a Heisenberg rigid motion, with the generating curve  on the $x_{n}t$-plane as follows:
\begin{equation*}
\begin{split}
x(s)&=e^{\int\bar\alpha ds}\\
t(s)&=-\int \bar{k}x^2 ds.
\end{split}
\end{equation*}
Notice that the generating curve is uniquely defined up to a translation along the $t$-axis.\\

Therefore, together with {\bf Theorem I} and {\bf Theorem B} in \cite{CCHY2}, we give a complete description of umbilic hypersurfaces.
From {\bf Theorem I}, there is an immediate corollary, please see Theorem \ref{funthm2}.\\
 
In addition, if the $p$-mean curvature $H$ of a rotationally symmetric hypersurface is constant, there is  an associated invariant for $\Sigma$, which is called the energy $E$. This energy $E$ is a constant and was first introduced by M. Ritor\'{e} and C. Rosales in their paper \cite{RR} to study the rotationally symmetric hypersurfaces in the Heisenberg groups. Actually, M. Ritor\'{e} and C. Rosales showed that the generating curve $\gamma=(x,t)$ of a rotationally symmetric hypersurface with constant $p$-mean curvature $H$ satisfies a system of ordinary differential equations whenever $x>0$. Moreover, the system they studied has a first integral that is constant along any solution. M. Ritor\'{e} and C. Rosales called the constant $E$ the energy of the solution and showed that the first integral allows them to give a complete description of the solution.  It turned out that they used the sign of the energy $E$, the $p$-mean curvature $H$, together with the sign of the product $EH$, to divide the generated constant $p$-mean hypersurfaces into six classes (for the details, see Theorem 5.4 in \cite{RR}). In this paper, in some sense, we showed that such hypersurfaces are in one-to-one correspondence with the energy $E$. We are also able to describe its generating curve as specified by the following {\bf Theorem II} and {\bf Theorem III}, which are our next two main theorems in this paper. Actually, in subsection \ref{maross}, we introduce the concept of the reflective rotationally symmetric hypersurfaces. Basically, each rotationally symmetric hypersurface locally is part of a reflective one. Theorems {\bf II} and {\bf III} point out that they are eventually in one-to-one correspondence with the energy $E$. In subsection \ref{energy}, we will give another realization of the energy $E$. This realization will help us to write out the generating curve of a reflective one. \\

{\bf Theorem II.}
For rotationally symmetric hypersurfaces with constant $p$-mean curvature $H=c>0$, the energy $E$ has a lower bound as follows:
\begin{equation}\label{ineqoE1}
E\geq\frac{-1}{2n}\left(\frac{2n-1}{c}\right)^{2n-1}.
\end{equation}
This lower bound is optimal. For any such value $E$, there exists a unique reflective rotationally symmetric hypersurface with constant $p$-mean curvature $H=c$ with $E$ as the energy. The corresponding generating curve is
\begin{equation}
\begin{split}
x(\tilde{s})&=F_{E}(\tilde{s})\\
t(\tilde{s})&=-\int \left(\frac{c}{2n}-\frac{E}{x^{2n}}\right) x^{2}d\tilde{s}.
\end{split}
\end{equation}
Here, the horizontal arc-length $\tilde{s}=G_{E}(x)$ is defined by the anti-derivative of the function with respect to $x$,
\[\pm \frac{1}{\sqrt{1-x^{2}\left(\frac{E}{x^{2n}}-\frac{c}{2n}\right)^{2}}} ,\]
for all $x\in(a,b)$, an interval depending on $E$. And $F_{E}(\tilde{s})$ is the inverse of $G_{E}(x)$. \\

When the constant $c<0$, {\bf Theorem II} has a duality description in which we just replace the optimal lower bound of $E$ by the optimal upper bound
\[E\leq \frac{1}{2n}\left(\frac{2n-1}{|c|}\right)^{2n-1}.\]
The proof of the duality is the same as the one for {\bf Theorem II}.\\
 
{\bf Theorem III.}
Given any $E\in\R$, there exists a unique reflective rotationally symmetric hypersurface with $p$-mean curvature $H=0$ with $E$ as the energy. The corresponding generating curve is
\begin{equation}
\begin{split}
x(\tilde{s})&=F_{E}(\tilde{s})\\
t(\tilde{s})&=\int \left(\frac{E}{x^{2n-2}}\right) d\tilde{s}.
\end{split}
\end{equation}
Here, the horizontal arc-length $\tilde{s}=G_{E}(x)$ is defined by the anti-derivative of the function with respect to $x$, 
\[\pm \frac{1}{\sqrt{1-\left(\frac{E}{x^{2n-1}}\right)^{2}}},\]
for all $x>|E|^{\frac{1}{2n-1}}$. And $F_{E}(\tilde{s})$ be the inverse of $G_{E}(x)$.\\

As applications, we can obtain some partial results in Alexandrov's Theorem and Bernstein's Problem (\cite{ACV,CHMY,DGNP,DGNP1}) in the Heisenberg groups $H_{n}$, see Corollary \ref{co1}, Corollary \ref{co2}, and Corollary \ref{co3}.  Moreover, we'd like to point out that {\bf Theorem II} and {\bf Theorem III} are, respectively, the generalizations of Theorem A and Theorem B in \cite{CLL} to the cases of general CR-dimension $n$. Finally, we would like to point out that in \cite{CLiu}, we have given a characterization description for constant $p$-mean curvature surfaces in the Heisenberg group $H_{1}$. They are just in one-to-one correspondence with the set of all $\alpha$ functions in some sense.\\

{\bf Acknowledgments.} The first author's research was supported in part by NSTC 112-2115-M-007-009-MY3. The second author's research was supported in part by NSTC 112-2115-M-167-002-MY2 and NSTC 114-2115-M-167-001-. The third author's research was supported in part by NSTC 112-2628-M-032-001-MY4.

\section{Umbilic hypersurfaces}\label{umhys}
In this section, we make a brief introduction to the basic materials of umbilic hypersurfaces in the Heisenberg group $H_{n}$ for $n\geq 2$. For the details, we refer the reader to \cite{CCHY1,CCHY2}.
Throughout this paper, we always assume $\Sigma $ is immersed
(say, $C^{2}$ smooth and $C^{\infty }$ smooth in the region of regular
points, see the definition of regular point below), and let $\xi $
($J,$ resp.) denote the standard contact ($CR$, resp.) structure on $H_{n},$
defined by the kernel of the contact form
\begin{equation*}
\Theta =dt+\sum_{j=1}^{n}(x_{j}dy_{j}-y_{j}dx_{j})
\end{equation*}

\noindent where $x_{1},..,x_{n},y_{1},..,y_{n},t$ are coordinates of $H_{n}.$
A point $p$ $\in $ $\Sigma $ is called singular if $\xi $ $=$ $T\Sigma $ at $p.$ Otherwise, $p$ is called regular or nonsingular. For a regular point, we define $\xi ^{\prime }$ $\subset $ $\xi \cap T\Sigma $ by
\begin{equation*}
\xi ^{\prime }=(\xi \cap T\Sigma )\cap J(\xi \cap T\Sigma ).
\end{equation*}

\noindent Let ($\xi ^{\prime }$)$^{\bot }$ denote the space of vectors in $\xi $, perpendicular to $\xi ^{\prime }$ with respect to the Levi metric $G:=\frac{1}{2}d\Theta (\cdot ,J\cdot )=\sum_{j=1}^{n}[(dx_{j})^{2}+(dy_{j})^{2}].$ It is not hard to see that $\dim (\xi 
$ $\cap $ $T\Sigma )$ $\cap $ ($\xi ^{\prime }$)$^{\bot }$ $=$ $1.$ Take $e_{n}$ $\in $ ($\xi $ $\cap $ $T\Sigma )$ $\cap $ ($\xi ^{\prime }$)$^{\bot
} $ of unit length. Define the horizontal normal $e_{2n}:=Je_{n}$. Let 
$\nabla $ denote the pseudohermitian connection associated to $(J,\Theta ).$
Let $\alpha $ be the function on $\Sigma $, such that $\alpha e_{2n}+T$ $\in 
$ $T\Sigma $ where $T$ :$=$ $\frac{\partial }{\partial t}.$ Define $J^{\prime }$ on $\xi \cap T\Sigma $ by
\begin{equation*}
J^{\prime }=J\text{ on }\xi ^{\prime }\text{ and }J^{\prime }e_{n}=0.
\end{equation*}

\noindent We now have a symmetric shape operator $\mathfrak{S}$ $:$ $\xi
\cap T\Sigma \rightarrow \xi \cap T\Sigma ,$ defined by
\begin{equation}
\mathfrak{S(}v\mathfrak{)}=-\nabla _{v}e_{2n}+\alpha J^{\prime }v.
\label{1-1}
\end{equation}
(see Proposition 2.2 in \cite{CCHY1}). A regular point $p$ is
called an umbilic point if $\mathfrak{S}(\xi ^{\prime })$ $\subset $ $\xi
^{\prime }$ at $p$ and all the eigenvalues of $\mathfrak{S}$ acting on $\xi
^{\prime }$ are the same. Let us denote this common eigenvalue by $k.$ On the other hand, by Proposition 2.3
in \cite{CCHY1}, we have
\begin{equation*}
\mathfrak{S}(e_{n})=le_{n}
\end{equation*}
 for some real number $l.$ In terms of $k$ and $l,$ we have
\begin{equation*}
H=l+(2n-2)k
\end{equation*}
at an umbilic point. 

In \cite{CCHY1}, J.-H. Cheng, H.-L. Chiu, J.-F. Hwang, and P. Yang showed that a rotationally symmetric hypersurface is umbilic. They also study umbilic hypersurfaces with positive constant $p$-mean curvature. 
Among others, we show that any umbilic sphere of
positive constant $H$ in $H_{n},$ $n$ $\geq $ $2,$ is a Pansu sphere up to a
Heisenberg rigid motion. Let us recall what Pansu spheres are. For any $\lambda>0$, the Pansu sphere $S_{\lambda }$ is the union of the graphs of the functions $g$ and $-g$, where
\begin{equation}
g(z)=\frac{1}{2\lambda ^{2}}\left( \lambda |z|\sqrt{1-\lambda ^{2}|z|^{2}}+\cos ^{-1}{\lambda |z|}\right) ,\ \ \ |z|\leq \frac{1}{\lambda }.
\label{1.6}
\end{equation}
It~is~known~that $S_{\lambda }$ has $p$-mean
curvature $H=2n\lambda$ (see Section 2 in \cite{CCHY1} for instance).

\begin{exa} In \cite{CCHY1}, there are two umbilic hypersurfaces with $\alpha=0$ constructed in Example 3.4. Following the notation there, we are going to give one more example. Suppose the function $f(x,y)$ defining a $\Sigma^{*}\subset\R^{2n}$ only depends on $(x_{n},y_{n})$, that is, $f(x,y)=f(x_{n},y_{n})$ and $\nabla f\neq 0 $ on $\Sigma^{*}$. Then the function $u(x,y,t)=f(x,y)=0$ defines a hypersurface $\Sigma_{\Sigma^{*}}$ on $H_{n}$. The following computation shows that the surface $\Sigma_{\Sigma^{*}}$ is umbilic with $\alpha=k=0$, and $l$ is the signed curvature of the curve $\gamma$ in the $x_{n}y_{n}$-plane, which is defined by $f(x_{n},y_{n})=0$. And
\[\Sigma_{\Sigma^{*}}=\C^{n-1}\times\gamma\times\R.\]
We have
\[e_{2n}=\frac{f_{n}\mathring{e}_{n}+f_{2n}\mathring{e}_{2n}}{\sqrt{(f_{n})^{2}+(f_{2n})^{2}}},\ \ \textrm{where}\ \ f_{n}=\frac{\partial f}{\partial x_{n}}, f_{2n}=\frac{\partial f}{\partial y_{n}}.\]
Since $e_{n}\perp\xi'$ and $e_{2n}\perp\xi'$, for any $e\in\xi'$, we have $e=\sum_{\beta=1}^{n-1}(a^{\beta}\mathring{e}_{\beta}+a^{n+\beta}\mathring{e}_{n+\beta})$. Therefore, 
\[-\nabla_{e}e_{2n}+\alpha J'e=-\nabla_{e}e_{2n}=0,\]
and hence $k=0$. Finally, we write $e_{2n}=a\mathring{e}_{n}+b\mathring{e}_{2n}$, and after a straightforward computation, we obtain
\[\begin{split}
-\nabla_{e_{n}}e_{2n}&=-e_{n}\left(\frac{f_{n}}{\sqrt{(f_{n})^{2}+(f_{2n})^{2}}}\right)\mathring{e}_{n}-e_{n}\left(\frac{f_{2n}}{\sqrt{(f_{n})^{2}+(f_{2{}n})^{2}}}\right)\mathring{e}_{2n}\\
&=-(b\mathring{e}_{n}-a\mathring{e}_{2n})a\mathring{e}_{n}-(b\mathring{e}_{n}-a\mathring{e}_{2n})b\mathring{e}_{2n}\\
&=(aa_{2n}-ba_{n})\mathring{e}_{n}+(ab_{2n}-bb_{n})\mathring{e}_{2n}\\
&=\frac{-(f_{n}^{2}f_{2n 2n}+f_{2n}^{2}f_{nn})+2f_{n}f_{2n}f_{n 2n}}{((f_{n})^{2}+(f_{2n})^{2})^{3/2}}e_{n},
\end{split}\]
where
\[f_{nn}=f_{x_{n}x_{n}},\ f_{n2n}=f_{x_{n}y_{n}},\ f_{2n2n}=f_{y_{n}y_{n}}.\]
Thus, 
\[l=\frac{-(f_{n}^{2}f_{2n 2n}+f_{2n}^{2}f_{nn})+2f_{n}f_{2n}f_{n 2n}}{((f_{n})^{2}+(f_{2n})^{2})^{3/2}},\]
which is just the signed curvature of $\gamma$ in the $x_{n}y_{n}$-plane.
\end{exa}

In \cite{CCHY2}, it shows that any umbilic hypersurface with $\alpha=k=0$ locally is part of $\C^{n-1}\times\gamma\times\R$, that is, locally there exist an open subset $U\subset\C^{n-1}$ and an open interval $I\in\R$ such that it is $U\times\gamma\times I$ for some plane curve $\gamma$ in the $x_{n}y_{n}$-plane, up to a Heisenberg rigid motion. On the other hand, if $\alpha^{2}+k^{2}>0$, then it is rotationally symmetric after an action of a left translation in the Heisenberg group. This result is just {\bf Theorem B} in \cite{CCHY2} (we present the whole description of {\bf Theorem B} in \cite{CCHY2} in the Introduction Section of this paper), which gave a partial classification for umbilic hypersurfaces.  Together with {\bf Theorem I}, which we provided here, we give a complete classification for umbilic hypersurfaces.

\section{Rotationally symmetric hypersurfaces}\label{sorss}
In this section, we will focus on the study of the rotationally symmetric hypersurfaces in the Heisenberg groups $H_{n}$ for $n\geq 1$. The proof of {\bf Theorem I} will be given at the end of this section.

\subsection{The formulae for $\alpha$, $k$, and $l$.}
Recall that if $\alpha=0$ for a rotationally symmetric hypersurface, then it is the cylinder. If $\alpha\neq 0$, then it can be presented as  the graph of a function that only depends on the radius variable $|z|$. We inherit the notation in \cite{CCHY1} and thus suppose that $t^2=f=f(r)$ is a function of $r, (r=|z|^2)$, then $u=f(r)-t^{2}$ is a defining function of the hypersurface $\Sigma$. 
With respect to 
the horizontal normal
\begin{equation}\label{honor}
e_{2n}=\frac{\nabla_{b}u}{|\nabla_{b}u|}
\end{equation} so that $e_{n}=-Je_{2n}$, we have in \cite{CCHY1} that
\begin{equation}\label{goinf1}
\begin{split}
\alpha&=\frac{t}{|z|\sqrt{(f')^2+f}},\\
k&=\frac{-f'}{|z|\sqrt{(f')^2+f}},\\ 
l&=\frac{|z|^2-f'}{|z|\sqrt{(f')^2+f}}-\frac{(1+2f'')|z|f}{((f')^2+f)^{3/2}},
\end{split}
\end{equation}
where in this subsection $f'(r)=\frac{df}{dr}$ and $f''(r)=\frac{d^2f}{dr^2}$.  Note that we have chosen this horizontal normal $e_{2n}$ to make sure that the Pansu sphere has positive $p$-mean curvature. Actually, formulae \eqref{goinf1} also hold  in the dimension $n=1$, while, instead of being an eigenvalue of the shape operator, the function $k$ has an interesting geometric meaning (see Proposition \ref{gemeok} below). 
After a straightforward computation, we see that
\begin{equation}\label{ori01}
\begin{split}
e_{2n}&=\sum_{\beta=1}^{n}\frac{(f'x_{\beta}-ty_{\beta})\mathring{e}_{\beta}+(f'y_{\beta}+tx_{\beta})\mathring{e}_{n+\beta}}{|z|\sqrt{(f')^{2}+f}},\\
e_{n}&=\sum_{\beta=1}^{n}\frac{(f'y_{\beta}+tx_{\beta})\mathring{e}_{\beta}-(f'x_{\beta}-ty_{\beta})\mathring{e}_{n+\beta}}{|z|\sqrt{(f')^{2}+f}}.\\
\end{split}
\end{equation}
We can also write out $e_{n}$ as
\begin{equation}\label{ori02}
e_{n}=\frac{\sum_{\beta=1}^{n}[(f'y_{\beta}+tx_{\beta})\frac{\partial}{\partial x_{\beta}}-(f'x_{\beta}-ty_{\beta})\frac{\partial}{\partial x_{\beta}}]+(f'r^2)\frac{\partial}{\partial t}}{|z|\sqrt{(f')^{2}+f}}.
\end{equation}

\subsection{The reflective symmetric hypersurface}\label{maross}
The hypersurface $\Sigma$ is divided into the union of two parts, where
$\Sigma=\Sigma^{+}\cup\Sigma^{-}$,
\begin{equation}
\begin{split}
\Sigma^{+}&=\{(z,t)\in\Sigma\ |\ t=\sqrt{f}\},\\
\Sigma^{-}&=\{(z,t)\in\Sigma\ |\ t=-\sqrt{f}\}.
\end{split}
\end{equation}
These two parts might be disjoint from each other. If we use $\alpha^{\pm},\ k^{\pm}$, and $\l^{\pm}$ to denote the corresponding geometric invariants for $\Sigma^{\pm}$, then from \eqref{goinf1}, for all $r> 0$ in the domain of $f$, we have that both $\alpha$ and $t$ have the same sign, and both $k$ and $f'$ have the opposite sign. And
\begin{equation}\label{relge}
\alpha^{+}(r)\geq 0,\ \ \ \alpha^{+}(r)=-\alpha^{-}(r),\ \ \ k^{+}(r)=k^{-}(r),\ \ \ l^{+}(r)=l^{-}(r).
\end{equation}
These two parts can be transformed into each other by the reflection with respect to the hyperplane $t=0$. Thus, $\Sigma$ is also a reflective symmetric hypersurface. For the properties of a rotationally symmetric hypersurface, it suffices to study such reflective symmetric hypersurfaces. The generating curve $\gamma$ for a reflective hypersurface $\Sigma$ will consist of two parts, $\gamma^{\pm}$, one for $\Sigma^{+}$, the other for $\Sigma^{-}$.
\subsection{The formulae for $\alpha$, $k$, and $l$ in terms of the generating curve.}
In terms of the generating curve $\gamma(s)=(0,\cdots,0,x(s),0,\cdots,0,t(s))$, which is lying on the $x_{n}t$-plane, the corresponding horizontal generating curve lying in the $x_{n}y_{n}t$-space is
\begin{equation}\label{hoge01}
\tilde{\gamma}(s)=(0,\cdots,0,x(s)\cos{\theta(s)},0,\cdots,0,x(s)\sin{\theta(s)},t(s)),
\end{equation}
for some angle function $\theta=\theta(s)$. Notice that the angle function $\theta(s)$ makes the curve $\tilde{\gamma}(s)$ to be horizontal, i.e., the velocity $\tilde{\gamma}'(s)\in\xi$ if it satisfies the following formula 
\begin{equation}\label{foang}
\theta'(s)=-\frac{t'}{x^2}(s),
\end{equation}
for each $s$. Here, $s$ need not be the arc-length. In this subsection, the symbol $'$ means taking the derivative with respect to $s$, instead of $r$. Comparing formulae \eqref{ori02} and \eqref{hoge01}, we see that $\frac{d\tilde{\gamma}}{ds}$ defines the same orientation as $e_{n}$ if and only if $t'=\frac{dt}{ds}$ and $\frac{df}{dr}$ have the same sign. From now on, throughout this paper, we always assume that $\frac{d\tilde{\gamma}}{ds}$ defines the same orientation as $e_{n}$.
\begin{prop}
With respect to the orientation defined by
$\frac{d\tilde{\gamma}}{ds}$, we have
\begin{equation}\label{goinf2}
\begin{split}
\alpha&=\frac{x'}{\sqrt{(xx')^2+(t')^2}}(s)\\
k&=\frac{-t'}{x\sqrt{(xx')^2+(t')^2}}(s)\\  
l&=\frac{-(t')^3-x^3(x't''-x''t')}{x(\sqrt{(xx')^2+(t')^2})^3}(s).
\end{split}
\end{equation}
Notice that $\alpha$ and $x'$ have the same sign, while $k$ and $t'$ have the opposite sign.
\end{prop}
\begin{proof}
It suffices to assume that $\frac{d\tilde{\gamma}}{ds}$ defines the same orientation as $e_{n}$.
Notice that $t^{2}=f(r)$ and $ r=|z|^{2}=x^{2}$, we have 
\[\frac{df}{dr}=\frac{dt^2}{dr}=2t\frac{dt}{dr}=2t\frac{dt}{ds}/\frac{dr}{ds},\]
where \[\frac{dr}{ds}=2|z|\frac{d|z|}{ds}=2x\frac{dx}{ds},\] and thus 
\begin{equation}\label{bas01}
\frac{df}{dr}=\frac{tt'}{xx'}.
\end{equation}
Therefore, we have
\[(\frac{df}{dr})^2+f=4t^{2}\frac{(t')^2}{(r')^2}+t^2=t^{2}\frac{(xx')^2+(t')^2}{(xx')^2},\] and 
\[1+2\frac{d^2f}{dr^2}=1+2\frac{d(\frac{df}{dr})}{ds}\frac{ds}{dr}=\frac{xx'+(\frac{tt'}{xx'})'}{xx'}.\]
Substituting the above two formulae into \eqref{goinf1}, we get \eqref{goinf2}, in which we have used the property that $\frac{df}{dr}$ and $t'$, and hence also $t$ and $x'$  have the same sign by \eqref{bas01}.
\end{proof}

\begin{rem}
It is easy to see that the formulae \eqref{goinf2} also hold for the generating curve $\tilde{\gamma}$ with the opposite orientation $-e_{n}$, and thus hold for the cylinder.
\end{rem}

From the first two equations \eqref{goinf2}, we immediately have
\begin{equation}\label{fogecu}
\alpha^2+k^2=\frac{1}{x^2}.
\end{equation}

In terms of horizontal arc-length $\tilde{s}$, with 
\begin{equation}\label{jacob}
\frac{ds}{d\tilde{s}}=\frac{x}{\sqrt{(xx')^2+(t')^2}},
\end{equation}
formulae \eqref{goinf2} are transformed into \eqref{goinf3} as follows.
 
\begin{prop}\label{gemeok}
\begin{equation}\label{goinf3}
\begin{split}
\frac{dx}{d\tilde{s}}&=\alpha x,\\  
\frac{dt}{d\tilde{s}}&=-kx^2,\ \left(\textrm{or}\ \ \frac{d\theta}{d\tilde{s}}=k\right)\\
l&=k+\left(\alpha\frac{dk}{d\tilde{s}}-k\frac{d\alpha}{d\tilde{s}}\right)x^2.
\end{split}
\end{equation}
\end{prop}
\begin{proof}
Using \eqref{foang}, it is trivial to see that the first two equations of \eqref{goinf3} are equivalent to the first two of \eqref{goinf2}. 
For the third one, we first note that
\[t''=\frac{d}{ds}\left(\frac{dt}{ds}\right)=\frac{d}{ds}\left(\frac{dt}{d\tilde{s}}\frac{d\tilde{s}}{ds}\right)=\frac{d^2t}{d\tilde{s}^2}\left(\frac{d\tilde{s}}{ds}\right)^{2}+\frac{dt}{d\tilde{s}}\left(\frac{d^2\tilde{s}}{ds^2}\right).\]
Similarly, we have
\[x''=\frac{d^2x}{d\tilde{s}^2}\left(\frac{d\tilde{s}}{ds}\right)^{2}+\frac{dx}{d\tilde{s}}\frac{d}{ds}\left(\frac{d\tilde{s}}{ds}\right).\]
The above two formulae imply that
\begin{equation}
\begin{split}
x't''-x''t'&=\frac{dx}{d\tilde{s}}\left(\frac{d\tilde{s}}{ds}\right)t''-x''\frac{dt}{d\tilde{s}}\left(\frac{d\tilde{s}}{ds}\right)\\
&=\left(\frac{dx}{d\tilde{s}}\frac{d^2t}{d\tilde{s}^2}\right)\left(\frac{d\tilde{s}}{ds}\right)^{3}+\frac{dx}{d\tilde{s}}\frac{dt}{d\tilde{s}}\frac{d^2\tilde{s}}{ds^2}\frac{d\tilde{s}}{ds}-\left(\frac{d^2x}{d\tilde{s}^2}\frac{dt}{d\tilde{s}}\right)\left(\frac{d\tilde{s}}{ds}\right)^{3}-\frac{dx}{d\tilde{s}}\frac{dt}{d\tilde{s}}\frac{d^2\tilde{s}}{ds^2}\frac{d\tilde{s}}{ds},\\
&=\left(\frac{dx}{d\tilde{s}}\frac{d^2t}{d\tilde{s}^2}-\frac{d^2x}{d\tilde{s}^2}\frac{dt}{d\tilde{s}}\right)\left(\frac{d\tilde{s}}{ds}\right)^{3},
\end{split}
\end{equation}
and also
\[(t')^3=\left(\frac{dt}{d\tilde{s}}\right)^3\left(\frac{d\tilde{s}}{ds}\right)^{3}.\]
Therefore, from the third equation of \eqref{goinf2} and using \eqref{jacob}, we derive
\begin{equation}\label{fofol}
\begin{split}
l&=\frac{-(t')^3-x^3(x't''-x''t')}{x(\sqrt{(xx')^2+(t')^2})^3}\\
&=\frac{-\left(\frac{dt}{d\tilde{s}}\right))^3-x^3\left(\frac{dx}{d\tilde{s}}\frac{d^2t}{d\tilde{s}^2}-\frac{d^2x}{d\tilde{s}^2}\frac{dt}{d\tilde{s}}\right)}{x^4}.\\
\end{split}
\end{equation}
On the other hand, the first two equations of \eqref{goinf3} yield that
\[\frac{d^2 t}{d\tilde{s}^2}=-\left(\frac{dk}{d\tilde{s}}+2k\alpha\right)x^2,\ \ \textrm{and}\ \ \frac{d^2x}{d\tilde{s}^2}=\left(\frac{d\alpha}{d\tilde{s}}+\alpha^2\right)x.\]
Substituting the above into \eqref{fofol}, we obtain that
\begin{equation}
\begin{split}
l&=\left(k(\alpha^2+k^2)+\alpha\frac{dk}{d\tilde{s}}-k\frac{d\alpha}{d\tilde{s}}\right)x^2\\
&=k+\left(\alpha\frac{dk}{d\tilde{s}}-k\frac{d\alpha}{d\tilde{s}}\right)x^2.
\end{split}
\end{equation}
\end{proof}

We thus have the formula for the generating curve  $\gamma(\tilde{s})=(0,\cdots,0,x(\tilde{s}),0,\cdots,0,t(\tilde{s}))$,
and the generated surface $\tilde{\gamma}(\tilde{s},\theta)$ in the $x_{n}y_{n}t$-space is parametrized by
\[\tilde{\gamma}(\tilde{s},\theta)=(0,\cdots,0,x(s)\cos{(\theta+\theta(s))},0,\cdots,0,x(s)\sin{(\theta+\theta(s))},t(\tilde{s})),\]
This surface is the intersection of the generated hypersurface by $\gamma(\tilde{s})$ and the $x_{n}y_{n}t$-space.

\begin{prop}\label{basfo1}
 For any $n\geq 1$, we have
\begin{equation}\label{geme2k}
\frac{\partial\tilde{\gamma}}{\partial\theta}(\tilde{s},\theta)=x^{2}(ke_{n}+\alpha e_{2n}+T).
\end{equation}
\end{prop}
\begin{proof}
Taking the derivative of $\tilde{\gamma}$ with respect to $\theta$,
\[\begin{split}
\frac{\partial\tilde{\gamma}}{\partial\theta}&=(0,\cdots,0,-x(s)\sin{(\theta+\theta(s))},0,\cdots,0,x(s)\cos{(\theta+\theta(s))},t(\tilde{s}))\\
&=\left[-x(s)\sin{(\theta+\theta(s))}\right]\mathring{e}_{1}+\left[x(s)\cos{(\theta+\theta(s))}\right]\mathring{e}_{2}+x^{2}\frac{\partial}{\partial t}.
\end{split}\]
To obtain \eqref{geme2k}, it is equivalent to computing the inner products of $\frac{\partial\tilde{\gamma}}{\partial\theta}$ with $e_{n}$ and $\alpha e_{2n}+T$, respectively. By the rotationally symmetric property, we only do them for $\theta=0$. First, we have
\[\begin{split}
&e_{n}=\frac{\partial\tilde{\gamma}}{\partial\tilde{s}}\\
&=(x'\cos{\theta(\tilde{s})}-x\sin{\theta(\tilde{s})}\cdot\theta'(\tilde{s}))\frac{\partial}{\partial x_n}+(x'\sin{\theta(\tilde{s})}+x\cos{\theta(\tilde{s})}\cdot\theta'(\tilde{s}))\frac{\partial}{\partial y_n}+(t'(\tilde{s}))\frac{\partial}{\partial t}\\
&=(x'\cos{\theta(\tilde{s})}-x\sin{\theta(\tilde{s})}\cdot\theta'(\tilde{s}))\mathring{e}_{n}+(x'\sin{\theta(\tilde{s})}+x\cos{\theta(\tilde{s})}\cdot\theta'(\tilde{s}))\mathring{e}_{2n}+(x^2\theta'(\tilde{s})+t'(\tilde{s}))\frac{\partial}{\partial t}\\
&=(x'\cos{\theta(\tilde{s})}+\frac{t'}{x}\sin{\theta(\tilde{s})})\mathring{e}_{n}+(x'\sin{\theta(\tilde{s})}-\frac{t'}{x}\cos{\theta(\tilde{s})})\mathring{e}_{2n},
\end{split}\]
where for the last equality, we have used the formula $x^2\theta'(\tilde{s})+t'(\tilde{s})=0$ since $e_{n}$ is horizontal. Therefore,
\[\begin{split}
\big<\frac{\partial\tilde{\gamma}}{\partial\theta},e_{n}\big>&=\left[-x(s)\sin{(\theta+\theta(s))}\right](x'\cos{\theta(\tilde{s})}+\frac{t'}{x}\sin{\theta(\tilde{s})})\\
&+\left[x(s)\cos{(\theta+\theta(s))}\right](x'\sin{\theta(\tilde{s})}-\frac{t'}{x}\cos{\theta(\tilde{s})})\\
&=-t'(\tilde{s}),\ \ (\textrm{notice that}\ \theta=0).
\end{split}\]
Similarly, 
\[\begin{split}
\alpha e_{2n}+T&=\alpha Je_{n}+T\\
&=\alpha(x'\cos{\theta(\tilde{s})}+\frac{t'}{x}\sin{\theta(\tilde{s})})\mathring{e}_{2n}-\alpha(x'\sin{\theta(\tilde{s})}-\frac{t'}{x}\cos{\theta(\tilde{s})})\mathring{e}_{n}+\frac{\partial}{\partial t}.\\
\end{split}\]
We have
\[\begin{split}
\big<\frac{\partial\tilde{\gamma}}{\partial\theta},\alpha e_{2n}+T\big>&=-\alpha\left[-x(s)\sin{(\theta+\theta(s))}\right](x'\sin{\theta(\tilde{s})}-\frac{t'}{x}\cos{\theta(\tilde{s})})\\
&+\alpha\left[x(s)\cos{(\theta+\theta(s))}\right](x'\cos{\theta(\tilde{s})}+\frac{t'}{x}\sin{\theta(\tilde{s})})\\
&=\alpha x x'+x^{2}\\
&=x^{2}(1+\alpha^{2}),\ \ \textrm{by}\ \eqref{goinf3}\ \ (\textrm{notice that}\ \theta=0).\\
\end{split}\]
We conclude that
\[\begin{split}
\frac{\partial\tilde{\gamma}}{\partial\theta}(\tilde{s},\theta)&=-t'e_{n}+x^{2}(\alpha e_{2N}+T)\\
&=x^{2}\left(-\frac{t'}{x^{2}}e_{n}+\alpha e_{2n}+T\right)\\
&=x^{2}(ke_{n}+\alpha e_{2n}+T),\ \ \textrm{by}\ \eqref{goinf3}.
\end{split}\]
This completes the proof.
\end{proof}

\begin{rem}\label{basfo}
Actually, in \cite{CCHY1,CCHY2}, J.-H. Cheng, H.-L. Chiu, J.-F. Hwang, and P. Yang had shown that for umbilic hypersurfaces, the vector field $k e_{n}+\alpha e_{2n}+T$ is spanned by $\xi'$ and $[\xi',\xi']$.
\end{rem}
We thus have the following identity formulae, which are just Proposition 4.2 in \cite{CCHY1}. Notice that, instead of using the integrability conditions, we recover these identities in terms of Proposition \ref{basfo1} or Remark \ref{basfo}.
\begin{prop}
\begin{equation}\label{reinfo}
\begin{split} 
\frac{dk}{d\tilde{s}}&=(l-2k)\alpha,\ \ \ \frac{d\alpha}{d\tilde{s}}=(k^2-kl-\alpha^2),\\
\hat{e}_{2n}k&=\frac{-k(e_{n}k)}{\sqrt{1+\alpha^2}},\ \ \ \hat{e}_{2n}\alpha=\frac{-k(e_{n}\alpha)}{\sqrt{1+\alpha^2}},\\
\hat{e}_{2n}l&=\frac{e_{n}e_{n}\alpha+6\alpha e_{n}\alpha+4\alpha^3+l^{2}\alpha}{\sqrt{1+\alpha^2}}.\\
\end{split}
\end{equation}
\end{prop}
\begin{proof}
Firstly, taking the derivatives of \eqref{fogecu} with respect to the horizontal arc-length $\tilde{s}$ gives
\[-\alpha=\left(k\frac{dk}{d\tilde{s}}+\alpha\frac{d\alpha}{d\tilde{s}}\right)x^2.\]
Together with the third equation of \eqref{goinf3}, we solve the system of equations to get the first two identities of \eqref{reinfo}.
To obtain the last three identities, we need Proposition \ref{basfo1} or Remark \ref{basfo}. By the first two identities of \eqref{reinfo}, we have
\[\begin{split}
e_{n}e_{n}\alpha&=e_{n}\left(\frac{d\alpha}{d\tilde{s}}\right)=e_{n}(k^2-kl-\alpha^2)\\
&=2ke_{n}k-le_{n}k-ke_{n}l-2\alpha e_{n}\alpha,\\
\end{split}\]
that is
\[\begin{split}
-ke_{n}l&=e_{n}e_{n}\alpha-2ke_{n}k+le_{n}k+2\alpha e_{n}\alpha\\
&=e_{n}e_{n}\alpha+2\alpha e_{n}\alpha+(l-2k)e_{n}k\\
&=e_{n}e_{n}\alpha+6\alpha e_{n}\alpha+4\alpha^3+l^{2}\alpha.
\end{split}\]
Therefore, by Proposition \ref{basfo1}, we get
\[(k e_{n}+\alpha e_{2n}+T)k=0,\ (k e_{n}+\alpha e_{2n}+T)\alpha=0,\ (k e_{n}+\alpha e_{2n}+T)l=0,\]
that is
\[(\alpha e_{2n}+T)k=-k e_{n}k,\ (\alpha e_{2n}+T)\alpha=-k e_{n}\alpha,\]
and
\[\begin{split}
(\alpha e_{2n}+T)l&=-k e_{n}l\\
&=e_{n}e_{n}\alpha+6\alpha e_{n}\alpha+4\alpha^3+l^{2}\alpha.
\end{split}\]
This completes the proof.
\end{proof}

\subsection{Fundamental theorems for rotationally symmetric hypersurfaces}\label{ftrss} In this subsection, we will give a proof of {\bf Theorem I}.
Let $\Sigma$ be a rotationally symmetric hypersurface in $H_{n}$ with generating curve 
\[(0,\cdots,0,x(s),0,\cdots,0,0,t(s)) \] on the $x_{n}t$-plane, then using the horizontal arc-length $\tilde{s}$, by \eqref{goinf3}, we have 
\begin{equation}\label{regecu}
\begin{split}
x(\tilde{s})&=\int \alpha x(\tilde{s})d\tilde{s}\ \ \left(\Leftrightarrow x(\tilde{s})=e^{\int\alpha d\tilde{s}}\right)\\
t(\tilde{s})&=-\int kx^2(\tilde{s})d\tilde{s},
\end{split}
\end{equation}
with \[x(\tilde{s})^{2}=\frac{1}{\alpha^2+k^{2}}(\tilde{s}),\] where $\tilde{s}$ is the horizontal arc-length of the corresponding horizontal generating curve.
Conversely, given two functions $\bar\alpha(s),\bar{k}(s)$, we define a curve on the $x_{n}t$-plane by
\begin{equation}\label{regecu1}
\begin{split}
x(s)&=e^{\int\bar\alpha ds},\\
t(s)&=-\int \bar{k}x^2 ds,
\end{split}
\end{equation}
here $x(s)$ is determined up to a constant multiple, and thus the curve is determined up to a Heisenberg dilation. Note that parameter $s$ here need not be the arc-length. We normalize $(x(s),t(s))$ by means of the condition
\begin{equation}\label{nomal}
\frac{(xx')^{2}+(t')^2}{x^2}=1.
\end{equation}
This curve is therefore uniquely defined up to a translation along the $t$-axis, and generates a rotationally symmetric hypersurface in $H_{n}$. The following theorem (Theorem \ref{funthm1}) shows that this generated hypersurface is just the rotationally symmetric one with $\bar\alpha(s)$ and $\bar{k}(s)$ as its $\alpha$-function and $k$-function.

\begin{thm}\label{funthm1}
Given two functions $\bar\alpha(s),\bar{k}(s)$ which satisfy the integrability condition \eqref{nomal}.
Then there exists a rotationally symmetric hypersurface such that $\alpha(s)=\bar\alpha(s)$ and $k(s)=\bar k(s)$ for all $s$. And $s$ is the horizontal arc-length. In addition, such a hypersurface is unique, up to a Heisenberg rigid motion.
\end{thm}
\begin{proof}
We consider the horizontal generating curve 
\[\tilde{\gamma}(s)=(0,\cdots,0,x(s)\cos{\theta(s)},0,\cdots,0,x(s)\sin{\theta(s)},t(s)).\]
for some angle function $\theta=\theta(s)$. After a straightforward computation, we have
\[\begin{split}
\frac{d\tilde{\gamma}}{ds}(s)&=(x'\cos{\theta}-x\sin{\theta}\cdot\theta')\frac{\partial}{\partial x_n}+(x'\sin{\theta}+x\cos{\theta}\cdot\theta')\frac{\partial}{\partial y_n}+(t')\frac{\partial}{\partial t}\\
&=(x'\cos{\theta}-x\sin{\theta}\cdot\theta')\mathring{e}_{n}+(x'\sin{\theta}+x\cos{\theta}\cdot\theta')\mathring{e}_{2n}+(x^2\theta'+t')\frac{\partial}{\partial t},
\end{split}\]
which implies that $\tilde{\gamma}(s)$ is horizontal if and only if 
\begin{equation*}
\theta'=-\frac{t'}{x^2}.
\end{equation*}
Thus, we have
\[\begin{split}
\left\|\frac{d\tilde{\gamma}}{ds}\right\|^{2}&=(x'\cos{\theta}-x\sin{\theta}\cdot\theta')^{2}+(x'\sin{\theta}+x\cos{\theta}\cdot\theta')^{2}\\
&=(x')^{2}+x^{2}(\theta')^{2}=\frac{(xx')^2+(t')^2}{x^2},
\end{split}\]
which shows that the normalization \eqref{nomal} just means that $s$ is the horizontal arc-length $\tilde{s} $ of the corresponding horizontal generating curve. Next, from the definition of the generating curve \eqref{regecu1}, we immediately have 
\begin{equation}\label{goinf4}
\bar\alpha=x^{-1}x'\ \ \textrm{and}\ \  \bar k=-x^{-2}t'. 
\end{equation}
Thus,
\[\begin{split}
\bar\alpha^{2}+\bar k^{2}&=x^{-2}(x')^{2}+x^{-4}(t')^{2}\\
&=\frac{1}{x^2}\left(\frac{(xx')^{2}+(t')^{2}}{x^{2}}\right)=\frac{1}{x^2}.
\end{split}\]
Let $\alpha$ and $k$ be the $\alpha$-function and $k$-function of the generated surface. Then \eqref{regecu} and \eqref{regecu1} yield that
\[(\alpha^2+k^{2})(\tilde{s})=(\bar\alpha^2+\bar k^{2})(s)\] and
\begin{equation}\label{regecu2}
\begin{split}
t(\tilde{s})&=-\int kx^2(\tilde{s})d\tilde{s}=-\int \bar kx^2(s)ds\\
&=-\int \bar kx^2(s)\frac{ds}{d\tilde{s}}d\tilde{s},
\end{split}
\end{equation}
for all $\tilde{s}$.
Therefore, we have
\begin{equation}
\begin{split}
k(\tilde{s})&=\bar k(s)\frac{ds}{d\tilde{s}},\\
\alpha^2(\tilde{s})&=\bar\alpha^2(s)+\bar k^{2}\left(1-(\frac{ds}{d\tilde{s}})^2\right).
\end{split}
\end{equation}
In particular, we have $k(\tilde{s})=\bar k(s)$ and $\alpha^2(\tilde{s})=\bar\alpha^2(s)$ as $\frac{ds}{d\tilde{s}}=1$, i.e, $s$ is the horizontal arc-length of the corresponding horizontal generating curve. Since \[x(s)=e^{\int \alpha ds}=e^{\int\bar\alpha ds},\] we have that $\alpha(\tilde{s})=\bar\alpha(s)$. This completes the proof of Theorem \ref{funthm1} and thus provides a proof for {\bf Theorem I}.
\end{proof}

We have the following immediate consequence.
\begin{thm}\label{funthm2}
Suppose $\bar\alpha(s)$ is a function such that
\begin{equation}\label{nomal1}
\frac{1}{(e^{\int\bar\alpha ds})^2}-\bar\alpha^{2}\geq 0,
\end{equation} 
for some anti-derivative $\int\bar\alpha ds$. We define a function $\bar{k}(s)$ by
\[\bar{k}^2=\frac{1}{(e^{\int\bar\alpha ds})^2}-\bar\alpha^{2}.\]
Then there exists a rotationally symmetric hypersurface such that $\alpha(s)=\bar\alpha(s)$ and $k(s)=\bar k(s)$ for all $s$. And $s$ is the horizontal arc-length. In addition, such a hypersurface is determined up to a Heisenberg rigid motion.
\end{thm}
\begin{proof}
We define the generating curve $(x(s),t(s))$ by
\[x(s)=e^{\int\bar\alpha ds},\ \ t(s)=-\int\bar{k}x^{2}ds.\]
We then claim that this curve satisfies the condition \eqref{nomal} as the following:
\[\frac{dx}{ds}=\bar{\alpha}e^{\int\bar\alpha ds}=\bar{\alpha}x\ \ \textrm{and}\ \ \ \frac{dt}{ds}=-\bar{k}x^{2}.\]
Thus,
\[\begin{split}
(xx')^{2}+(t')^{2}&=\bar\alpha^{2}x^{4}+\bar{k}^{2}x^{4}\\
&=\bar\alpha^{2}x^{4}+\left(-\bar\alpha^{2}+\frac{1}{(e^{\int\bar\alpha ds})^2}\right)x^4\\
&=x^{2}.
\end{split}
\] This completes the proof. 
\end{proof}
We remark that since $\bar{k}$, defined in Theorem \ref{funthm2}, is determined up to a sign, there are two determined hypersurfaces, with respect to different signs. Therefore, from the definition of the generating curve, we see that these two hypersurfaces differ by a reflection about the $xy$-plane.

\begin{exa}
If $\bar\alpha=0$, then $\bar k=\pm const$. This is the vertical cylinder, which is symmetric with respect to the reflection about the $xy$-plane. 
\end{exa}

\begin{exa}
We compute the example
\[t=\frac{\sqrt{3}}{2}x^{2}.\]
Consider the generating curve in the $xt$-plane $(x,t)=(x,\frac{\sqrt{3}}{2}x^{2})$. The corresponding horizontal generating curve, with parameter $x$, is
\[\tilde{\gamma}(x)=(0,\cdots,0,x\cos{\theta(x)},0,\cdots,0,x\sin{\theta(x)},t(x))\] 
The velocity
\[\frac{d\tilde{\gamma}(x)}{dx}=(\cos{\theta}-x\sin{\theta}\cdot\theta')\mathring{e}_{n}+(\sin{\theta}+x\cos{\theta}\cdot\theta')\mathring{e}_{2n}+(x^2\theta'+t')\frac{\partial}{\partial t}.\]
Therefore, $\tilde{\gamma}(x)$ is horizontal if and only if $\theta'=\frac{-t'}{x^2}=\frac{-\sqrt{3}}{x}$.
We thus get \[\theta=-\sqrt{3}\ln{x}.\]
It is easy to see that $\|\frac{d\tilde{\gamma}(x)}{dx}\|^2=4$, and hence $x=\frac{1}{2}\tilde{s},\ t=\frac{\sqrt{3}}{8}\tilde{s}^2$. From \eqref{goinf3}, we have that 
\[\alpha=\frac{1}{\tilde{s}}\ \ \textrm{and}\ \ \ k=\frac{-\sqrt{3}}{\tilde{s}}.\]
Conversely, given $\alpha=\frac{1}{\tilde{s}}$, then $x(\tilde{s})$ is determined up to a positive constant multiple. That is
\[x(\tilde{s})=e^{\int\alpha d\tilde{s}}=a\tilde{s},\]
for some positive constant $a\in\R$. Therefore the normal condition \eqref{nomal1} holds if and only if $0<a\leq 1$. We can then define the function $k$ by $k^{2}=\frac{1}{x^2}-\alpha^2=\left(\frac{1-a^2}{a^2}\right)\frac{1}{\tilde{s}^2}$, i.e.,
\[k=\pm\frac{\sqrt{1-a^2}}{a}\frac{1}{\tilde{s}}.\]
Then
\[\frac{dt}{d\tilde{s}}=-kx^2=\mp a\sqrt{1-a^2}\tilde{s}\]
or
\[t=\mp \frac{a\sqrt{1-a^2}}{2}\tilde{s}^2=\mp \frac{\sqrt{1-a^2}}{2a}x^2,\]
up to an addition constant, for $0<a\leq 1$. Finally, from the third formula \eqref{goinf3}, it is easy to obtain that $l=k$ for all $0<a\leq 1$.
\end{exa}

\section{Rotationally symmetric hypersurfaces with constant $p$-mean curvature}
Suppose that $\Sigma$ is a rotationally symmetric hypersurface with constant $p$-mean curvature. In subsection \ref{energy}, we will show that for each such kind of hypersurface, there is associated a constant invariant, which is actually the first integral $E$ of an ODE system introduced by M. Retor\'e and C. Rosales in \cite{RR}.

\subsection{The energy $E$}\label{energy}
In this subsection, we will give another realization of the energy $E$. This realization will help us to give proofs of {\bf Theorem II} and {\bf Theorem III}.
We write the $p$-mean curvature $H=(2n-2)k+l=2n\lambda$ for some $\lambda$, that is, $l-2k=2n(\lambda-k)$.
We observe both the first equations of \eqref{goinf3} and \eqref{reinfo} and notice that
\[(l-2k)\frac{dx}{d\tilde{s}}-x\frac{dk}{d\tilde{s}}=0.\]
Suppose that $\lambda$ is constant, the above equation is thus transformed into
\[2n(\lambda-k)\frac{dx}{d\tilde{s}}+x\frac{d(\lambda-k)}{d\tilde{s}}=0.\]
Multiplying $x^{2n-1}$ on both sides, we have
\[\begin{split}
0&=2n(\lambda-k)x^{2n-1}\frac{dx}{d\tilde{s}}+x^{2n}\frac{d(\lambda-k)}{d\tilde{s}}\\
&=\frac{d}{d\tilde{s}}\left[(\lambda-k)x^{2n}\right].
\end{split}\]
That is, $(\lambda-k)x^{2n}$ is a constant. After a straightforward computation, we have
\[\begin{split} 
(\lambda-k)x^{2n}&=x^{2n-2}\frac{dt}{d\tilde{s}}+\lambda x^{2n}\\
&=\frac{x^{2n-1}t'}{\sqrt{(xx')^{2}+(t')^2}}+\lambda x^{2n},
\end{split}\]
which is just the first integral $E$  of an ODE system introduced by M. Retor\'e and C. Rosales in \cite{RR}. Here $
'$ means $\frac{d}{ds}$.
Conversely, suppose that the energy $E$ is constant. If we follow the above arguments backwards, then it is easy to see that $\lambda$ is constant.
In summary, we have the following propositions for rotationally symmetric hypersurfaces.
\begin{prop}\label{chaofcon}
The $p$-mean curvature $H$ is constant if and only if the energy $E$ is constant.
\end{prop}
Recall that in Section \ref{sorss}, we divide $\Sigma$ into the union of two parts, $\Sigma=\Sigma^{+}\cup\Sigma^{-}$. From \eqref{relge}, \eqref{fogecu}, and the formula for $E$, we immediately have the following proposition.
\begin{prop}\label{chaofcon1}
$\Sigma^{+}$ is of constant $p$-mean curvature if and only if $\Sigma^{-}$ is of constant $p$-mean curvature. And they have the same energy $E$.
\end{prop}

\subsection{Constant $p$-mean curvature surfaces}
Suppose that $H=(2n-2)k+l=2n\lambda=c$, we have the energy formula
\begin{equation}\label{fomuE1}
E=(\lambda-k)x^{2n},\end{equation}
or
\begin{equation}\label{fomuE2}
k=\lambda-\frac{E}{x^{2n}}=\frac{c}{2n}-\frac{E}{x^{2n}}.
\end{equation}
Thus, by \eqref{fogecu} and \eqref{goinf3}, we have
\begin{equation}
\begin{split}
1&=x^{2}(\alpha^{2}+k^{2})=x^{2}\alpha^{2}+x^{2}k^{2}\\
&=(x')^{2}+x^{2}k^{2},
\end{split}
\end{equation}
that is,
\begin{equation}\label{reofafd}
(x')^{2}=1-x^{2}\left(\frac{E}{x^{2n}}-\frac{c}{2n}\right)^{2}, 
\end{equation}
or
\begin{equation}
x'=\pm\sqrt{1-x^{2}\left(\frac{E}{x^{2n}}-\frac{c}{2n}\right)^{2}}.
\end{equation}
Notice that we have chosen the orientation $e_{n}$ so that the sign of $x'$ is the same as the one of $\alpha$ (see \eqref{goinf3}).
We solve the above equation to get
\begin{equation}\label{solox}
\begin{split}
\pm \tilde{s}+c_{1}&=\int\frac{dx}{\sqrt{1-x^{2}\left(\frac{E}{x^{2n}}-\frac{c}{2n}\right)^{2}}},\\
&=\int\frac{x^{2n-1}dx}{\sqrt{x^{4n-2}-\left(E-\frac{c}{2n}x^{2n}\right)^{2}}}.
\end{split}
\end{equation}
for some constant $c_{1}$. The function $x(\tilde{s})$ is just the inverse of the function $\tilde{s}$ in \eqref{solox}. Having the $x(\tilde{s})$, we in addition to solve the equation $t'=-kx^{2}$ to get the generating curve $(x(\tilde{s}),t(\tilde{s}))$, which depends only on the energy $E$. Notice that the sign $\pm$ is corresponding to the generating curve of the individual part $\Sigma^{\pm}$ of surface $\Sigma$.
{\bf Theorem II} declares that $E$ is the only invariant for rotationally symmetric hypersurfaces with constant $p$-mean curvature $H=c$.

\subsection{The proof of {\bf Theorem II}}

We first prove the lower bound of $E$ for the cases $c>0$. From \eqref{fomuE2}, and notice that $(\alpha^{2}+k^{2})x^{2}=1$, we have that
\[\left(\frac{E}{x^{2n}}-\frac{c}{2n}\right)^{2}\leq\frac{1}{x^{2}},\]
that is,
\begin{equation}\label{ineqoE}
\frac{c}{2n}x^{2n}-x^{2n-1}\leq E\leq\frac{c}{2n}x^{2n}+x^{2n-1},\ \ \textrm{for all}\ \tilde{s}.
\end{equation}
Let $f:\R\rightarrow\R$ be a function defined by the left polynomial in \eqref{ineqoE},
\[f(x)=\frac{c}{2n}x^{2n}-x^{2n-1}.\]
It is a straightforward computation that
\[\frac{df}{dx}=(cx-(2n-1))x^{2n-2},\ \ \ \frac{d^{2}f}{dx^{2}}=(cx-(2n-2))(2n-1)x^{2n-3}.\]
Thus, the point $x=\frac{2n-1}{c}$ is the unique critical point at which the second derivative $\frac{d^{2}f}{dx^{2}}(\frac{2n-1}{c})$ is positive. We conclude that the function $f(x)$ has the absolute minimum at $x=\frac{2n-1}{c}$, that is,
\[f(x)=\frac{c}{2n}x^{2n}-x^{2n-1}\geq\frac{-1}{2n}\left(\frac{2n-1}{c}\right)^{2n-1}.\]
Example \ref{optex} later shows that this bound is optimal. Next, we prove the existence and uniqueness. Given any value $E$, which is strictly greater than the optimal lower bound. From the first inequality in \eqref{ineqoE}, there exists an interval $(a,b)$ such that
\[-1<\frac{2nE-cx^{2n}}{2nx^{2n-1}}<1,\ \ \ \textrm{for}\ x\in(a,b),\]
which is equivalent to the inequality
\[1-x^{2}\left(\frac{E}{x^{2n}}-\frac{c}{2n}\right)^{2}>0,\ \ \ \textrm{for}\ x\in(a,b).\]
Define a function $\tilde{s}$ on $(a,b)$ by $\tilde{s}=G_{E}(x)$, where $G_{E}(x)$ is the anti-derivative of 
\[\pm \frac{1}{\sqrt{1-x^{2}\left(\frac{E}{x^{2n}}-\frac{c}{2n}\right)^{2}}},\]
where the sign ($+$ or $-$) allows us to define two generating curves later. We have $\frac{dG_{E}}{dx}\neq 0$, which implies that $G_{E}$ has the inverse function. Let $F_{E}(\tilde{s})$ be the inverse of $G_{E}(x)$. We define $x(\tilde{s})=F_{E}(\tilde{s})$ and then define
\[\alpha(\tilde{s})=\frac{x'}{x}(\tilde{s}),\ \ \ k(\tilde{s})=\frac{c}{2n}-\frac{E}{x^{2n}}.\]
We have 
\[\alpha(\tilde{s})=\frac{x'}{x}(\tilde{s})=\frac{F_{E}'}{x}=\frac{1}{x\frac{dG_{E}}{dx}}=\frac{\sqrt{1-x^{2}\left(\frac{E}{x^{2n}}-\frac{c}{2n}\right)^{2}}}{x},\]
and hence
\[\alpha^{2}+k^{2}=\frac{1}{x^{2}},\]
which is equivalent to the integrability condition \eqref{nomal} in Theorem \ref{funthm1}. By Theorem \ref{funthm1}, these two functions, $\alpha$ and $k$, determine a unique rotationally symmetric hypersurface that is also reflective (due to the sign $\pm$, which, respectively, defines the generating curves $\gamma^{+}$ and $\gamma^{-}$ for $\Sigma^{+}$ and $\Sigma^{-}$).  In addition, from the definition of $k$, we have that $(\lambda-k)x^{2n}$ is just the constant $E$. Therefore, from the proof of Proposition \ref{chaofcon}, we conclude that the mean curvature $H$ is constant. This completes the proof of {\bf Theorem II}.

\subsection{The proof of {\bf Theorem III}}

There exists a number $\delta>0$ such that
\[-1<\frac{E}{x^{2n-1}}<1,\ \ \ \textrm{for}\ x>\delta,\]
which is equivalent to the inequality
\[1-\left(\frac{E}{x^{2n-1}}\right)^{2}>0,\ \ \ \textrm{for}\ x>\delta.\]
Define a function $\tilde{s}$ on $x>\delta$ by $\tilde{s}=G_{E}(x)$, where $G_{E}(x)$ is the anti-derivative of 
\[\pm \frac{1}{\sqrt{1-\left(\frac{E}{x^{2n-1}}\right)^{2}}},\]
where the sign ($+$ or $-$) allows us to define two generating curves later. We have that $\frac{dG_{E}}{dx}\neq 0$, which implies that $G_{E}$ has the inverse function. Let $F_{E}(\tilde{s})$ be the inverse of $G_{E}(x)$. We define $x(\tilde{s})=F_{E}(\tilde{s})$ and then define
\[\alpha(\tilde{s})=\frac{x'}{x}(\tilde{s}),\ \ \ k(\tilde{s})=-\frac{E}{x^{2n}}.\]
We have 
\[\alpha(\tilde{s})=\frac{x'}{x}(\tilde{s})=\frac{F_{E}'}{x}=\frac{1}{x\frac{dG_{E}}{dx}}=\frac{\sqrt{1-\left(\frac{E}{x^{2n-1}}\right)^{2}}}{x},\]
and thus
\[\alpha^{2}+k^{2}=\frac{1}{x^{2}},\]
which is equivalent to the integrability condition \eqref{nomal} in Theorem \ref{funthm1}. By Theorem \ref{funthm1}, these two functions, $\alpha$ and $k$, determine a unique rotationally symmetric hypersurface, which is also reflective (Due to the sign $\pm$, which, respectively, define the generating curves $\gamma^{+}$ and $\gamma^{-}$ for $\Sigma^{+}$ and $\Sigma^{-}$).  In addition, from the definition of $k$, we have that $-kx^{2n}$ is just the constant $E$. Therefore, from the proof of Proposition \ref{chaofcon}, we conclude that the mean curvature $H$ is constant and must be zero. This completes the proof of {\bf Theorem III}.

\begin{exa}\label{optex}
In this example, we will show that the lower bound for the energy $E$ is optimal.\\
(i)The energy of the Pansu sphere is $E=0$.\\
(ii) For the cylinder, we have $\alpha=0$ and $l=k$, so that $1=x^{2}(\alpha^{2}+k^{2})=x^{2}k^{2}$ and $c=(2n-2)k+l=(2n-1)k$. 
Therefore
\[\begin{split}
E&=(\lambda-k)x^{2n}=\left(\frac{c}{2n}-\frac{c}{2n-1}\right)\frac{1}{\left(\frac{c}{2n-1}\right)^{2n}}\\
&=\frac{-1}{2n}\left(\frac{2n-1}{c}\right)^{2n-1},
\end{split}\]
which is the lower bound of the energy $E$ in \eqref{ineqoE1}.
\end{exa}

As applications, we can obtain some partial results in Alexandrov's Theorem and Bernstein's Problem in the Heisenberg groups $H_{n}$.

\begin{cor}\label{co1}
Let $\Sigma$ be an embedded and closed hypersurface in $H_{n}$ with constant $p$-mean curvature $H=c>0$. If, moreover, it is umbilic with a singular point, then it is the Pansu sphere after a Heisenberg rigid motion. 
\end{cor}
\begin{proof} 
Since $\Sigma$ is closed, the function $\alpha$ is not identically zero. Thus, {\bf Theorem B} in \cite{CCHY2} implies that it is rotationally symmetric. Since a singular point of a rotationally symmetric hypersurface only happens at the $t$-axis, which means that the $x$-component of the generating curve will ever get closer to $0$. 
From \eqref{reofafd}, we have
\begin{equation*}
0\leq (x')^{2}=1-x^{2}\left(\frac{E}{x^{2n}}-\frac{c}{2n}\right)^{2}, 
\end{equation*}
which implies that 
\[-2nx^{2n-1}+cx^{2n}\leq 2nE\leq 2nx^{2n-1}+cx^{2n}.\]
Since $x$ will ever get closer to $0$, we see that the energy of $\Sigma$ is $E=0$. By {\bf Theorem II}, it is the Pansu sphere.
\end{proof}
Actually, Corollary \ref{co1} had been proved in \cite{CCHY1,CCHY2} in different ways. The following corollary shows that the condition of the existence of a singular point is not necessary. 

\begin{cor}\label{co2}
Let $\Sigma$ be an embedded and closed hypersurface in $H_{n}$ with constant $p$-mean curvature $H=c>0$. If, moreover, it is umbilic, then it is the Pansu sphere after a Heisenberg rigid motion. 
\end{cor}
\begin{proof}
Since $\Sigma$ is closed, the function $\alpha$ is not identically zero. Thus, {\bf Theorem B} in \cite{CCHY2} implies that it is rotationally symmetric. We will show that the $x$-component of the generating curve will ever get closer to $0$. For a contradiction, we assume that it does not, then the generating curve will be a closed curve in the $x_{n}t$-plane, then $x(0)=x(L)$ and $t(0)=t(L)$ for some $L>0$. Since $t'=-kx^{2}$, we have that
\begin{equation}\label{basicfor}
0=-\int_{0}^{L} t' d\tilde{s}=\int_{0}^{L} kx^{2} d\tilde{s}.
\end{equation}
On the other hand, \eqref{reofafd} suggests
\begin{equation*}
(x')^{2}=1-x^{2}k^{2}=1-x^{2}\left(\frac{c}{2n}-\frac{E}{x^{2n}}\right)^{2},
\end{equation*}
thus $(x'(0))^{2}=(x'(L))^{2}$, and hence that $x'(0)=x'(L)$ by smoothness of the generating curve. Also,
\begin{equation*}
\int_{0}^{L}(x')^{2} d\tilde{s}=\int_{0}^{L}1-x^{2}k^{2} d\tilde{s}=L-\int_{0}^{L}x^{2}k^{2} d\tilde{s}.
\end{equation*}
By integration by parts, 
\begin{equation*}
\int_{0}^{L}(x')^{2} d\tilde{s}=-\int_{0}^{L} xx'' d\tilde{s}=-\int_{0}^{L} x^{2}((2n-1)k^{2}-ck) d\tilde{s},
\end{equation*}
where, for the last equality, we have used the following
\[x''=(\alpha x)'=\alpha'x+\alpha^{2}x=(k^{2}-kl-\alpha^{2})x+\alpha^{2}x=((2n-1)k^{2}-ck)x.\]
Comparing the above two integrals of $(x')^2$ and use \eqref{basicfor}, we obtain that
\[L-\int_{0}^{L}x^{2}k^{2} d\tilde{s}=-\int_{0}^{L} x^{2}((2n-1)k^{2}-ck) d\tilde{s}=-(2n-1)\int_{0}^{L} x^{2}k^{2} d\tilde{s},\]
that is,
\[L=(2-2n)\int_{0}^{L} x^{2}k^{2} d\tilde{s}\leq 0,\ \ \textrm{for}\ n\geq1,\]
which is a contradiction. We thus conclude that the $x$-component of the generating curve will ever get closer to $0$. From \eqref{reofafd} again, we have
\begin{equation*}
0\leq (x')^{2}=1-x^{2}\left(\frac{E}{x^{2n}}-\frac{c}{2n}\right)^{2}, 
\end{equation*}
which implies that 
\[-2nx^{2n-1}+cx^{2n}\leq 2nE\leq 2nx^{2n-1}+cx^{2n}.\]
Since $x$ will ever get closer to $0$, we see that the energy of $\Sigma$ is $E=0$. By {\bf Theorem II}, it is the Pansu sphere.
\end{proof}

For Bernstein's problem, we have the following corollary, which holds for any CR dimension $n$.
\begin{cor}\label{co3}
Let $\Sigma$ be a $p$-minimal hypersurface in $H_{n}$ which is defined by the graph of a function $f$ defined on the whole space $R^{2n}$. If, moreover, it is umbilic, then it is a horizontal hyperplane after a Heisenberg rigid motion. 
\end{cor}
\begin{proof} 
It is defined by a graph, so the function $\alpha\neq0$. Thus, {\bf Theorem B} in \cite{CCHY2} implies that it is rotationally symmetric. 
From \eqref{reofafd}, we have
\begin{equation*}
0\leq (x')^{2}=1-x^{2}\left(\frac{E}{x^{2n}}\right)^{2}=1-\frac{E^{2}}{x^{4n-2}},
\end{equation*}
Since $f$ is defined on the  whole space $R^{2n}$, we see that $x$ will ever get closer to $0$, and thus that the energy of $\Sigma$ is $E=0$. By {\bf Theorem III}, it is a horizontal hyperplane.
\end{proof}

\subsection{The case for $n=1$.} In this subsection, we will show that {\bf Theorem II} and {\bf Theorem III} are just Theorem A and Theorem B in \cite{CLL} in the case $n=1$. That is, we generalize them to general CR dimension $n$.
If $c>0$, then
\begin{equation}
\begin{split}
\pm \tilde{s}+c_{1}&=\int\frac{xdx}{\sqrt{x^{2}-\left(E-\frac{c}{2}x^{2}\right)^{2}}}=\frac{1}{2}\int\frac{\frac{1}{\lambda}dy}{\sqrt{\frac{y+E}{\lambda}-y^{2}}}\ \ \ (y=\lambda x^{2}-E)\\
&=\frac{1}{2}\int\frac{\frac{1}{\lambda}dy}{\sqrt{\left(\frac{E}{\lambda}+\frac{1}{4\lambda^{2}}\right)-(y-\frac{1}{2\lambda})^{2}}}\\
&=\frac{1}{c}\int\frac{dy}{\sqrt{\left(\frac{2E}{c}+\frac{1}{c^{2}}\right)-(y-\frac{1}{c})^{2}}}
\end{split}
\end{equation}
or
\[\pm c\tilde{s}+cc_{1}=\int\frac{d\left(y-\frac{1}{c}\right)}{\sqrt{\left(\frac{2E}{c}+\frac{1}{c^{2}}\right)-(y-\frac{1}{c})^{2}}},\]
which implies that
\[\pm c\tilde{s}+cc_{1}=\sin^{-1}{\frac{\lambda x^{2}-(E+\frac{1}{c})}{a}},\ \ \ \mbox{where } a=\frac{\sqrt{2cE+1}}{c}.\]
Thus,
\[x^{2}=\frac{2\sqrt{2cE+1}}{c^{2}}\sin{(\pm c\tilde{s}+cc_{1})}+\frac{2cE+2}{c^{2}}.\]
If $c=0$, then
\[\begin{split}
\pm \tilde{s}+c_{1}&=\int\frac{xdx}{\sqrt{x^{2}-E^{2}}}\\
&=\frac{1}{2}\int\frac{dy}{\sqrt{y}},\ \ \ y=x^{2}-E^{2},
\end{split}\]
and thus
\[x^{2}=(\pm\tilde{s}+c_{1})^{2}+E^{2}.\]
For general CR dimension $n$, it seems difficult to compute the related integrals for the functions $\tilde{s}=G_{E}(x)$.


\begin{thebibliography}{99}

\bibitem{ACV} V. Barone Adesi, F. Serra Cassano, and D. Vittone, {\it The Bernstein problem for intrinsic graphs in Heisenberg groups and calibrations},  Preprint.

\bibitem{CHMY} Cheng, J.-H.; Hwang, J.-F.; Malchiodi, A., and Yang, P., {\it Minimal surfaces in Pseudohermitian geometry},  Annali della Scuola Normale Superiore di Pisa Classe di Scienze V , 4 (1), 129-177, 2005.

\bibitem{CCHY1} Cheng, J.-H.; Chiu, H.-L.; Hwang, J.-F., and Yang, P.,{\it Umbilicity and characterization of Pansu spheres in the Heisenberg group}, Journal fur die reine und angewandte Mathematik (738), 203-235, 2018.

\bibitem{CCHY2} Cheng, J.-H.; Chiu, H.-L.; Hwang, J.-F., and Yang, P.,{\it Umbilic hypersurfaces of constant sigma-k curvature in the Heisenberg group}, Calc. Var. Partial Differential Equations (55), 25 pages, 2016.

\bibitem{CFH} H. L. Chiu, X. Feng and Y. C. Huang, The differential geometry of curves in the Heisenberg groups. \textit{Differential Geom. Appl.} \textbf{56} (2018), 161--172.

\bibitem{CL} Chiu, H.-L. and Lai, S.-H., {\it The fundamental theorem for hypersurfaces in Heisenberg groups}, Calc. Var. Partial Diffrential Equations, 54 (2015), no. 1, 1091-1118.

\bibitem{CLiu} Hung-Lin Chiu and Hsiao-Fan Liu, \textit{A characterization of constant p-mean curvature surfaces in the Heisenberg group $H_1$}, Advances in Mathematics 405 (2022) 108514.

\bibitem{CHL} Chiu, H.-L.; Huang, Y.-C.; Lai, S.-H, {\it An Application of the Moving Frame Method to Integral Geometry in the Heisenberg Group},  Symmetry, Integrability and Geometry: Methods and Applications, 13 (2017), 097, 27 pages.

\bibitem{CLL} Chiu, H.-L., Lai, S.-H. and Liu, H.-F.; {\it On invariants of constant $p$-mean curvature surfaces in the Heisenberg group $H_{1}$}, Symmetry, Integrability and Geometry: Methods and Applications, 21 (2025), 011, 25 pages.

\bibitem{DGNP} D. Daniekki, N. Garofalo, D. M. Nhieu and S. D. Pauls, {Instability of graphical strips and a positive answer to the Bernstein problem in the Heisenberg group $H^{1}$\it}, J. Differential Geometry, 81, pp 251--295, 2009.

\bibitem{DGNP1} D. Danielli, N. Garofalo, DM Nhieu and S., Pauls, {\it The Bernstein problem for Embedded surfaces in the Heisenberg group $H_{1}$},  Indiana University mathematics journal, pp 563-594, 2010.

\bibitem{DT}  Dragomir, S., Tomassini, G., {\it Critical Point Theory and Submanifold Geometry}, Springer, Berlin (1988)

\bibitem{J} Jacobowitz, H., {\it An Introduction to CR structures}, Mathematical Surveys and Monographs, no. 32 (1990).
in Mathematics, vol. 166. Springer-Verlag, New York (1997)

\bibitem{M} Malchiodi, A., {\it Minimal surfaces in three dimensional pseudo-Hermitian geometry},  Lecture Notes of Seminario Interdisciplinare di Matematica, Vol. 6, pp. 195-207, 2007.

\bibitem{RR} M. Ritor\'{e} and C. Rosales, {\it Rotationally invariant Hypersurfaces with constant mean curvature in the Heisenberg group $H^{n}$}, The Journal of Geometric Analysis, v.16, n.4 pp. 703-720 (2006).

\bibitem{P} S. D. Pauls, {\it Minimal surfaces in the Heisenberg group}, Geometriae Dedicata,104, pp. 201-231, 2004.

\bibitem{PT}  Palais, R., Terng, C.-L., {\it Critical Point Theory and Submanifold Geometry}, Springer, Berlin (1988).

\bibitem{S} Sharpe,R.W., {\it Differential geometry: Cartan's generalization of Klein,s Erlangen program}, Graduate Texts
in Mathematics, vol. 166. Springer-Verlag, New York (1997).

\bibitem{W} Webster, S.M., {\it Pseudo-Hermitian structures on a real Hypersurface}, J. Differ. Geom. 13 (1978), 25-41.

\bibitem{Y}Yang, P., {\it Pseudo-Hermitian geometry in 3D},  Geometric analysis, 113-144, Lecture Notes in Math., 2263, 
Fond. CIME/CIME Found. Subser., Springer, Cham, (2020).

\end{thebibliography}
\bibliographystyle{plain}

\end{document}